\documentclass[11pt,a4paper,twoside]{article}
\usepackage{amsthm, amsfonts,amsmath}

\newtheorem{theorem}{Theorem}[section]
\newtheorem{corollary}{Corollary}[section]
\newtheorem{definition}{Definition}[section]
\newtheorem{example}{Example}[section]
\newtheorem{lemma}{Lemma}[section]
\newtheorem{proposition}{Proposition}[section]
\newtheorem{remark}{Remark}[section]

\author{
        Dhriti Sundar Patra\footnote{Department of Mathematics, University of Haifa, Mount Carmel, 3498838 Haifa, Israel
       \newline e-mail:
       {\tt dhritimath@gmail.com}}
        \ and \
       Vladimir Rovenski\footnote{Department of Mathematics, University of Haifa, Mount Carmel, 3498838 Haifa, Israel
       \newline e-mail: {\tt vrovenski@univ.haifa.ac.il}
       }
}

\title{On the rigidity of the Sasakian structure and characterization of cosymplectic manifolds}

\begin{document}

\date{}

\maketitle

\begin{abstract}
We introduce new metric structures on a smooth manifold (called ``weak" structures)
that generalize the almost contact, Sasakian, cosymplectic, etc. metric structures $(\varphi,\xi,\eta,g)$
and allow us to take a fresh look at the classical theory.
We demonstrate this statement by generalizing several well-known results.
We prove that any Sasakian structure is rigid, i.e., our weak Sasakian structure is homothetically equivalent to a~Sasakian structure.
We show that a weak almost contact structure with parallel tensor $\varphi$ is a weak cosymplectic structure
and give an example of such a structure on the product of manifolds.
We find conditions for a vector field to be a weak contact infinitesimal transformation.


%Secondary %53C21
\end{abstract}

%\setcounter{secnumdepth}{4}
%%%%%%%%%%%%%%%%%%%%%%%%%%%%%%%%%%%%%%%%%%

%\begin{document}
%%%%%%%%%%%%%%%%%%%%%%%%%%%%%%%%%%%%%%%%%%

%\setcounter{page}{1}

\section*{Introduction}

The methods of contact geometry play an important role in modern mathematics and theoretical physics.
For a review of the theory of contact metric structure and related structures, see \cite{blair2010riemannian,bg,Cappelletti-Montano2013,sasaki1965almost}.
There are inclusions $C_3\subset C_2\subset C_1$ and $C_4\subset C_1$ of well-known classes of metric structures on a manifold: almost contact $C_1$, contact $C_2$, Sasakian $C_3$ and cosymplectic $C_4$, respectively.

We introduce and study new metric structures on a smooth manifold (called ``weak" structures, see particular case in \cite{RWo-2} and \cite[Section~5.3.8]{Rov-Wa-2021})
that generalize the almost contact, Sasakian, cosymplectic, etc. metric structures $(\varphi,\xi,\eta,g)$
and allow us to take a fresh look at the theory of these classical structures.
We demonstrate this statement by generalizing several well-known results
on Sasakian and cosymplectic metric structures and contact vector fields.

Our metric structures: weak almost contact $C^{\,w}_1$, weak contact $C^{\,w}_2$, weak Sasakian $C^{\,w}_3$ and weak cosymplectic $C^{\,w}_4$,
form wider classes than the classical structures, i.e., their identity opera\-tor is replaced by a nonsingular $(1,1)$-tensor $Q$, see \eqref{2.1} in what follows.
Using the homothetical equiva\-lence relation, we define the equivalence classes $C^{\,w}_{i/\sim}$ of our structures with natural inclusions
$C^{\,w}_{4/\sim}\subset C^{\,w}_{1/\sim}$ and $C^{\,w}_{3/\sim}\subset C^{\,w}_{2/\sim}\subset C^{\,w}_{1/\sim}$.
Any classical structure $C_{i}$ mentioned above belongs to a family of weak structures that are homothetically equivalent to the classical structure; thus,
$C_i\subset C^{\,w}_{i/\sim}$.
A~natural question arises: \textit{How rich are these new structures $C^{\,w}_{i/\sim}$ compared to the classical ones~$C_{i}$}?
In~this article, we answer this question for the weak Sasakian and weak cosymplectic structures.
The study of the problem for weak contact metric manifolds and further generalization of classical results on contact metric manifolds postponed to the future.

This article consists of an introduction and six sections.
In Section~\ref{sec:1}, we define new metric structures $C^{\,w}_{i}$ and their homothetical equivalence classes $C^{\,w}_{i/\sim}$.
In Sections~\ref{sec:2} and \ref{sec:3a}, we study a weak contact metric structure and the tensor field $h$.
In Section~\ref{sec:3} we show that any weak Sasakian structure is homothetically equivalent to a Sasakian structure: $C_3= C^{\,w}_{3/\sim}$,
that is Sasakian structures are rigid in some sense.
In Section~\ref{sec:4}, we study $C^{\,w}_{4}$ and
show that a weak almost contact metric structure with parallel tensor $\varphi$ is a weak cosymplectic structure
and give an example of such a structure on the product of an even-dimensional manifold with a line or a~circle.
In Section~\ref{sec:5}, we find conditions for a vector field to be a weak contact vector field.

\section{Preliminaries}
\label{sec:1}

Here, we define new metric structures that generalize the almost contact metric structure.
A~\textit{weak almost contact structure} on a smooth odd-dimensional manifold $M^{2n+1}$ is a set $(\varphi,Q,\xi,\eta)$,
where $\varphi$ is a $(1,1)$-tensor, $\xi$ is the characteristic vector field and $\eta$ is a dual 1-form, satisfying
\begin{equation}\label{2.1}
\varphi^2 = -Q + \eta\otimes Q\,\xi, \quad \eta(\xi)=1,
\end{equation}
and $Q$ is a nonsingular $(1,1)$-tensor field such that
\begin{equation}\label{2.1Q-nu}
 Q\,\xi=\nu\,\xi
\end{equation}
for a positive constant $\nu$, see \cite{RWo-2} and \cite[Section~5.3.8]{Rov-Wa-2021} where $\nu=1$.
By $\eta(\xi)=1$, the form $\eta$ determines a smooth $2n$-dimensional contact distribution ${\cal D}:=\ker\eta$,
defined by the subspaces ${\cal D}_m=\{X\in T_m M: \eta(X)=0\}$ for $m\in M$.
We assume that the distribution
${\cal D}$ is $\varphi$-invariant,
%~i.e.,
\begin{equation}\label{2.1-D}
 \varphi X\in{\cal D},\quad X\in{\cal D},
\end{equation}
as in the classical theory of almost contact structure \cite{blair2010riemannian}, where $Q={\rm id}_{\,TM}$.
By \eqref{2.1} and \eqref{2.1-D}, the distribution ${\cal D}$ is invariant for $Q$: $Q({\cal D})={\cal D}$.
 If~there is a Riemannian metric $g$ on $M$ such that
\begin{align}\label{2.2}
 g(\varphi X,\varphi Y)= g(X,Q\,Y) -\eta(X)\,\eta(Q\,Y),\quad X,Y\in\mathfrak{X}_M,
\end{align}
then $(\varphi,Q,\xi,\eta,g)$ is called a {\it weak almost contact metric structure} on $M$ and $g$
is called a \textit{compatible} metric.
A weak almost contact manifold $M(\varphi,Q,\xi,\eta)$ endowed with a compatible Riemannian metric $g$
is said to be a \textit{weak almost contact metric manifold} and is denoted by $M(\varphi,Q,\xi,\eta,g)$.

Putting $Y=\xi$ in \eqref{2.2} and using $Q\,\xi=\nu\,\xi$ and $\nu\ne0$, we get, as in the classical theory,
\begin{align}\label{2.2A}
 \eta(X)=g(X,\xi).
\end{align}
In particular, $\xi$ is $g$-orthogonal to ${\cal D}$ for any compatible metric $g$.

\begin{remark}\rm
According to \cite{RWo-2}, a weak almost contact structure admits a compatible metric if $\varphi$ in \eqref{2.1}--\eqref{2.1-D}
has a skew-symmetric representation, i.e., for any $x\in M$ there exist a neighborhood $U_x\subset M$ and a~frame $\{e_i\}$ on $U_x$,
for which $\varphi$ has a skew-symmetric matrix.
\end{remark}

The following statement (a) generalizes \cite[Theorem~4.1]{blair2010riemannian}.

\begin{proposition}
%\label{P-6}
%Let $(\varphi,\xi,\eta,Q)$ be a {weak almost contact structure} on $M^{2n+1}$.
{\rm (a)} For a weak almost contact structure on $M$, the tensor $\varphi$ has rank $2n$ and
%the following equalities hold:
\[
 \varphi\,\xi=0,\quad \eta\circ\varphi=0,\quad [Q,\,\varphi]=0.
\]
{\rm (b)} For a weak almost contact metric structure, $\varphi$ is skew-symmetric and $Q$ is self-adjoint,
\begin{equation}\label{E-Q2-g}
 g(\varphi X, Y) = -g(X, \varphi Y),\quad
 g(QX,Y)=g(X,QY).
\end{equation}
\end{proposition}

\begin{proof}
(a) By \eqref{2.1}, $\varphi^2\xi=0$,
hence, either $\varphi\,\xi=0$ or $\varphi\,\xi$ is a nontrivial vector of $\ker\varphi$.
Applying \eqref{2.1} to $\varphi\,\xi$, we get $Q(\varphi\,\xi)=\eta(\varphi\,\xi)Q\,\xi$.
If $\varphi\,\xi=\mu\,\xi$ for a nonzero function $\mu:M\to\mathbb{R}$, then
$0=\varphi^2\xi =\mu\cdot\varphi\,\xi=\mu^2\xi\ne0$ -- a contradiction.
Assuming $\varphi\,\xi=\mu\,\xi+X$ for some $\mu:M\to\mathbb{R}$ and nonzero $X\in\ker\eta$,
by \eqref{2.1} we get $QX=0$ -- a contradiction to non-singularity of $Q$. Thus,~$\varphi\,\xi=0$.

Next, since $\varphi\,\xi=0$ everywhere, ${\rm rank}\,\varphi < 2n+1$.
If a vector field $\xi'$ satisfies $\varphi\,\xi'=0$, then \eqref{2.1} gives
$Q\,\xi'=\eta(\xi')Q\,\xi$. One may write $\xi'=\mu\xi+X$ for some $\mu:M\to\mathbb{R}$ and $X\in\ker\eta$.
This yields
$Q\,(\mu\xi+X)=\eta(\mu\xi+X)Q\,\xi$,
i.e., $Q X =0$, hence, $\xi'$ is collinear with $\xi$, and so ${\rm rank}\,\varphi = 2n$.

To show $\eta\circ\varphi=0$, note that $\eta(\varphi\,\xi)=\eta(0)=0$, and, using \eqref{2.1-D}, we get $\eta(\varphi X)=0$ for $X\in{\cal D}$.
Using \eqref{2.1} and $\varphi(Q\,\xi)=\nu\,\varphi\,\xi=0$, we get
\begin{align*}
 \varphi^3 X & = \varphi(\varphi^2 X) = -\varphi\,QX +\eta(X)\,\varphi(Q\,\xi) = -\varphi\,QX,\\
 \varphi^3 X & = \varphi^2(\varphi X) = -Q\,\varphi X +\eta(\varphi X) Q\,\xi = -Q\,\varphi X
\end{align*}
for any $X$, that proves $[Q,\,\varphi]=Q\,\varphi - \varphi Q = 0$.

(b) By~\eqref{2.2}, the~restriction $Q_{\,|{\cal D}}$ is self-adjoint. This and \eqref{2.1Q-nu} provide \eqref{E-Q2-g}$_2$.
For any $Y\in{\cal D}$ there is $\tilde Y\in{\cal D}$ such that $\varphi\,Y=\tilde Y$.
Thus, \eqref{E-Q2-g}$_1$ follows from \eqref{2.2} and \eqref{E-Q2-g}$_2$ for $X\in{\cal D}$ and $\tilde Y$.
\end{proof}

The {fundamental $2$-form} $\Phi$ on $M(\varphi,Q,\xi,\eta,g)$ is defined (as in the classical case) by
\begin{align*}
%\label{Phi-classic}
 \Phi(X,Y)=g(X,\varphi Y),\quad X,Y\in\mathfrak{X}_M.
\end{align*}
We~introduce a \textit{weak contact metric structure} as a weak almost contact metric structure satisfying
\begin{align}\label{2.3}
 \Phi=d\eta,
\end{align}
where
\begin{equation}\label{3.3A}
 d\eta(X,Y) = \frac12\,\{X(\eta(Y)) - Y(\eta(X)) - \eta([X,Y])\},\quad X,Y\in\mathfrak{X}_M.
\end{equation}
Recall the following formulas (with $X\in\mathfrak{X}_M$):
\begin{align}\label{3.3B}
 (\pounds_{\xi}\,\varphi)(X) &  = [\xi, \varphi X] - \varphi [\xi, X],\\
\label{3.3C}
 (\pounds_{\xi}\,\eta)(X) & = \xi(\eta(X)) - \eta([\xi, X]) ,
\end{align}
where $\pounds_{\xi}$ is the Lie derivative in the $\xi$-direction.
A weak almost contact metric structure $(\varphi,Q,\xi,\eta,g)$, for which $\xi$
is Killing, i.e., $\pounds_{\xi}\,g=0$, is called a \textit{weak $K$-contact structure}.

A weak almost contact structure $(\varphi,Q,\xi,\eta)$ on a manifold $M$ will be called {\it normal}
if the following tensor is identically zero:
\begin{align}\label{2.6X}
 N^{(1)}(X,Y) = [\varphi,\varphi](X,Y) + 2\,d\eta(X,Y)\,Q\,\xi,\quad X,Y\in\mathfrak{X}_M.
\end{align}
The Nijenhuis torsion $[\varphi,\varphi]$ of $\varphi$ is given~by
\begin{align}\label{2.5}
 [\varphi,\varphi](X,Y) = \varphi^2 [X,Y] + [\varphi X, \varphi Y] - \varphi[\varphi X,Y] - \varphi[X,\varphi Y],\quad X,Y\in\mathfrak{X}_M.
\end{align}

\begin{remark}\rm
Recall that the Levi-Civita connection $\nabla$ of a Riemannian metric $g$ is given by
\begin{align}\label{3.2}
 2\,g(\nabla_{X}Y,Z) &= X\,g(Y,Z) + Y\,g(X,Z) - Z\,g(X,Y) \notag\\
 &+ g([X,Y],Z) +g([Z,X],Y) - g([Y,Z],X),
\end{align}	
and has the properties $[X,Y]=\nabla_X\,Y-\nabla_Y\,X$ and $X\,g(Y,Z)=g(\nabla_X\,Y,Z)+g(Y,\nabla_X\,Z)$.
Thus, \eqref{2.5} can be written in terms of $\nabla\varphi$ as
\begin{align}\label{4.NN}
 [\varphi,\varphi](X,Y) = (\varphi\nabla_Y\,\varphi - \nabla_{\varphi Y}\,\varphi) X - (\varphi\nabla_X\,\varphi - \nabla_{\varphi X}\,\varphi) Y;
\end{align}
in particular, since $\varphi\,\xi=0$,
\begin{align}\label{4.NNxi}
 [\varphi,\varphi](X,\xi)= \varphi(\nabla_{\xi}\,\varphi)X +\nabla_{\varphi X}\,\xi -\varphi\,\nabla_{X}\,\xi, \quad X\in \mathfrak{X}_M,
\end{align}
\end{remark}

A normal weak contact metric manifold will be called a {\it weak Sasakian manifold}.

Next, we define an equivalence relation on the set of weak almost contact structures $C^{\,w}_1$ and on its important subsets of
weak contact metric structures $C^{\,w}_2$ and weak Sasakian structures $C^{\,w}_3$.
The factor-spaces are denoted by $C^{\,w}_{i/\sim}$ for $i=1,2,3$.
One can easily verify that the following definition is correct (see the proof of Proposition~\ref{P-22}).

\begin{definition}\rm
(i)~Two weak almost contact structures $(\varphi,Q,\xi,\eta)$ and $(\varphi',Q',\xi,\eta)$ on $M$ are said to be homothetically equivalent if
\begin{subequations}
\begin{align}
\label{Tran'}
 & \varphi = \sqrt\lambda\ \varphi', \\
\label{E-Q'-lambda}
 & Q\,|_{\,{\mathcal D}}=\lambda\,Q'\,|_{\mathcal D},\quad
 Q\,\xi = \lambda'\,Q'\xi
\end{align}
for some $\lambda,\lambda'\in\mathbb{R}_+$.
%%%
(ii)~Two weak contact metric structures $(\varphi,Q,\xi,\eta,g)$ and $(\varphi',Q',\xi,\eta,g')$ on $M$
are said to be homothetically equivalent if they satisfy conditions \eqref{Tran'}, \eqref{E-Q'-lambda} and
\begin{align}\label{Tran2'}
 g|_{\,{\mathcal D}} = \lambda^{-\frac12}\,g'|_{\,{\mathcal D}},\quad
 \iota_{\,\xi}\,g= \iota_{\,\xi}\,{g}' ,
\end{align}
where $\iota$ $($``\,iota"$)$ is the interior product operation.
(iii)~Two weak Sasakian structures $(\varphi,Q,\xi,\eta,g)$ and $(\varphi',Q',\xi,\eta,g')$ on $M$
are homothetically equivalent if they satisfy conditions (\ref{Tran'}-c) and
\begin{align*}
%\label{Tran3}
 \lambda = \lambda' .
\end{align*}
\end{subequations}
\end{definition}

In our modification of classical conditions we are motivated by the following.

\begin{proposition}\label{P-22}
Let $(\varphi,Q,\xi,\eta)$ be a weak almost contact structure on $M$ such that
\begin{subequations}
\begin{align*}
%\label{E-Q-lambda}
 Q\,|_{\,{\mathcal D}}=\lambda\,{\rm id}_{\mathcal D},\quad
 Q\,\xi = \nu\,\xi
\end{align*}
for some $\lambda,\nu\in\mathbb{R}_+$.
%Consider the transformation:
Then the following is true:

\noindent
{\rm (i)} $(\varphi', \xi, \eta)$ is an almost contact structure on $M$, where $\varphi'$ is given by
%$(\varphi,Q,\xi,\eta)$ is a weak almost contact structure on $M$, and
\begin{align}\label{Tran}
\varphi = \sqrt{\lambda}\ \varphi'.
\end{align}
{\rm (ii)} if $(\varphi,Q,\xi,\eta,g)$ is a weak contact metric structure on $M$ with conditions \eqref{Tran} and
\begin{align}\label{Tran2}
 g|_{\,{\mathcal D}} = \lambda^{-\frac12}\,g'|_{\,{\mathcal D}},\quad
 \iota_{\,\xi}\,g= \iota_{\,\xi}\,{g}',
\end{align}
then
$(\varphi',\xi,\eta,{g}')$ is a contact metric structure on $M$.

\noindent
{\rm (iii)} if $(\varphi,Q,\xi,\eta,g)$ is a weak Sasakian structure on $M$ with conditions {\rm (\ref{Tran},b)} and
\begin{align*}
%\label{Tran3}
 \lambda = \nu,
\end{align*}
\end{subequations}
then $(\varphi', \xi, \eta, {g}')$ is a Sasakian structure on $M$.
\end{proposition}

\begin{proof}{\rm (i)}
Substituting $\varphi = \gamma\,\varphi'$ in \eqref{2.1},
for $\xi$ we get identity, and for $X\in\mathcal{D}$ we get
\begin{align*}
%\label{K}
 \gamma^2(\varphi')^2 X = \lambda\,X,
\end{align*}
which reduces to the classical equality when $\gamma=\sqrt{\lambda}$.
Thus, $(\varphi', \xi, \eta)$ is an almost contact structure on $M$.

\smallskip

{\rm (ii)} Let $g|_{\,{\mathcal D}} = \beta\,g'|_{\,{\mathcal D}}$ and $\iota_{\,\xi}\,g=\delta\iota_{\,\xi}\,{g}'$
for some functions $\beta$ and $\delta$ on $M$. Using \eqref{2.2}, we get
\[
 \lambda\,\beta\,g'(\varphi' X, \varphi' Y) = \lambda\,\beta\, g'(X, Q'Y) -\lambda\,\eta(X)\,\eta(Q' Y)
 = \lambda\,\beta\,g'(X, Q' Y), \quad X,Y\in\mathcal{D},
\]
-- no restrictions for $\beta$ and $\lambda$. The same is when $X=\xi$ or/and $Y=\xi$.
Next, we calculate
\[
 \delta\, \iota_{\,\xi}\,{g}'(Y) \overset{\eqref{Tran2}}=
 \iota_{\,\xi}\,g(Y)=g(\xi,Y)\overset{\eqref{2.2A}}=\eta(Y) \overset{\eqref{2.2A}}= {g}'( {\xi},Y) = \iota_{\,\xi}\,{g}'(Y),
\]
 which gives $\delta=1$.
Next, substituting the values of $(\varphi,Q,\xi,\eta,g)$ in \eqref{2.3}, we get
\begin{align}\label{d-eta}
 d\eta(X,Y)=\beta\gamma\,g'(X, \varphi'\,Y).
\end{align}
For $ X,Y\in\mathcal{D}$, \eqref{d-eta} is valid for $\beta\gamma=1$,
i.e., $\beta=\frac{1}{\sqrt{\lambda}}$, and for $X=\xi$ or/and $Y=\xi$ we find $d\eta({\xi},Y)=0$, see Proposition~\ref{prop2.1} below.
Hence, $(\varphi', \xi, \eta, {g}')$ is a contact metric structure on $M$.

{\rm (iii)} Substituting the values of $(\varphi,Q,\xi,\eta,g)$ in the equality $N^{(1)}(X,Y)=0$ and using the pro\-perty of Lie bracket, we get
\[
 \lambda\,[\varphi', \varphi' ](X,Y) + 2\,\nu\,d\eta(X,Y)\,\xi=0
\]
for any $X,Y\in\mathfrak{X}_M$, which reduces to the classical case when $\nu=\lambda$.
Hence, $(\varphi', \xi, \eta, {g}')$ is a Sasakian structure on $M$.
\end{proof}

The three tensors $N^{(2)}, N^{(3)}$ and $N^{(4)}$ are well known in the classical theory, see \cite{blair2010riemannian}:
\begin{align}
\label{2.7X}
 N^{(2)}(X,Y) &= (\pounds_{\varphi X}\,\eta)(Y) - (\pounds_{\varphi Y}\,\eta)(X),  \\
\label{2.8X}
 N^{(3)}(X) &= (\pounds_{\xi}\,\varphi)X \overset{\eqref{3.3B}}= [\xi, \varphi X] - \varphi [\xi, X],\\
\label{2.9X}
 N^{(4)}(X) &= (\pounds_{\xi}\,\eta)(X) \overset{\eqref{3.3C}}= \xi(\eta(X)) - \eta([\xi, X])
 = 2\,d\eta(\xi, X).
\end{align}
Using \eqref{3.3A} and $\varphi\,\xi=0$, one can calculate \eqref{2.7X} explicitly as
%for a weak almost contact metric structure
\begin{align}\label{3.6}
 N^{(2)}(X,Y)
% &= (\pounds_{\varphi X}\,\eta)(Y) - (\pounds_{\varphi Y}\,\eta)(X) \notag\\
 &= (\varphi X)(\eta(Y)) - \eta([\varphi X,Y]) - (\varphi Y)(\eta(X)) + \eta([\varphi Y,X]) \notag\\
 &=2\,d\eta(\varphi X,Y) - 2\,d\eta(\varphi Y,X) .
\end{align}

\section{Weak contact metric manifolds}
\label{sec:2}

Here, we study the weak contact metric structure.
Define a ``small" (1,1)-tensor $\tilde Q$ by
\begin{equation}\label{E-tildeQ}
 \tilde{Q} = Q - {\rm id} ,
\end{equation}
and note that $[\tilde{Q},\varphi]=0$ and $\tilde{Q}\,\xi=(\nu-1)\,\xi$, where $\nu = g(Q\,\xi,\xi)$, see \eqref{2.1Q-nu}.
We also obtain
\begin{equation}\label{E2-tildeQ}
  -\tilde{Q}\,\varphi = \varphi+\varphi^3 .
\end{equation}
Note that $X^\top = X - \eta(X)\,\xi$ is the projection of the vector $X\in TM$ onto ${\cal D}$.

\begin{remark}\rm
Similarly to \eqref{2.1}--\eqref{2.2}, we can define a weak para-contact metric structure, and use $\tilde Q=Q-{\rm id}_{\,TM}$
instead of \eqref{E-tildeQ}, etc. (see \cite{RWo-2} and \cite[Section~5.3.8]{Rov-Wa-2021} where $\nu=-1$),
which allows us (as in Sections~\ref{sec:2}--\ref{sec:5}) generalize some results on para-contact metric structures introduced in~\cite{kw1985}.
\end{remark}

The following theorem generalizes \cite[Theorem~6.1]{blair2010riemannian}, i.e., $Q={\rm id}_{\, TM}$.

\begin{theorem}\label{thm6.1}
For a weak almost contact metric structure $(\varphi,Q,\xi,\eta,g)$, the vanishing of $N^{(1)}$ implies that $N^{(3)}$ and $N^{(4)}$ vanish and
\begin{align*}
%\label{3.1KK}
 N^{(2)}(X,Y) =\eta([\tilde QX^\top,\,\varphi Y]) +g(\tilde Q\,\xi,\xi)\,\eta([\varphi X,Y]) .
\end{align*}
\end{theorem}

\begin{proof}
%Claim: $N^{(4)}=0$.
By assumption: $N^{(1)}(X,Y)=0$. Thus, taking $\xi$ instead of $Y$ and using the expression of Nijenhuis tensor \eqref{2.5}, we obtain
\begin{align}\label{3.11}
 0 &= [\varphi,\varphi](X,\xi) + 2\,d\eta(X,\xi)\,Q\,\xi \notag\\
%\label{3.12}
 &= \varphi^2[X,\xi] - \varphi[\varphi X,\xi] + 2\,d\eta(X,\xi)\,Q\,\xi.
\end{align}
Taking the scalar product of \eqref{3.11} with $\xi$ and using
skew-symmetry of $\varphi$, $\varphi\,\xi=0$ and $\nu\ne0$, we~get
\begin{align}\label{3.11A}
 d\eta(X,\xi)=0,
\end{align}
hence, \eqref{2.9X} yields $N^{(4)}=0$.
%Claim: $N^{(3)}=0$.
Next, combining \eqref{3.11} and \eqref{3.11A}, we get
\begin{align*}
 %\label{3.13}
 0 &= [\varphi,\varphi](X,\xi)
%\notag\\	&
= \varphi^2[X,\xi] - \varphi[\varphi X,\xi] \notag\\
 &= \varphi\{(\pounds_{\xi}\,\varphi)X - \varphi\,\pounds_{\xi}X\}
 = \varphi\,(\pounds_{\xi}\,\varphi)X,
\end{align*}
by using \eqref{2.5} and $\varphi\,\xi=0$. Applying $\varphi$ and using \eqref{2.1} and $\eta\circ\varphi=0$, we achieve
\begin{align}
\label{3.14}
 0 &= \varphi^2 (\pounds_{\xi}\,\varphi)X
 = -Q(\pounds_{\xi}\,\varphi)X + \eta((\pounds_{\xi}\,\varphi)X)\,Q\,\xi \notag\\
 &= -Q(\pounds_{\xi}\,\varphi)X + \eta([\xi,\varphi X])\,Q\,\xi - \eta(\varphi [\xi,X])\,Q\,\xi \notag\\
 &=  -Q(\pounds_{\xi}\,\varphi)X + \eta([\xi,\varphi X])\,Q\,\xi.
\end{align}
Further, \eqref{3.11A} and \eqref{3.3A} yield that
\begin{align}\label{3.11B}
	0=2\,d\eta(\varphi X, \xi)
	=(\varphi X)(\eta(\xi)) - \xi(\eta(\varphi X)) - \eta([\varphi X, \xi])
	=\eta([\xi, \varphi X]).
\end{align}
Since $Q$ is non-singular, from \eqref{3.14} we get $\pounds_{\xi}\,\varphi=0$, i.e, $N^{(3)}=0$.
	
%Claim: $N^{(2)}(X,Y)=\eta([QX-X,\varphi Y])$.
 Replacing $X$ by $\varphi X$ in our assumption $N^{(1)}=0$ and using \eqref{2.5} and \eqref{3.3A}, we acquire
\begin{align}\label{2.6}
 0 &= g([\varphi,\varphi](\varphi X,Y) + 2\,d\eta(\varphi X,Y)\,Q\,\xi,\ \xi) \notag\\
 &= g([\varphi^2 X,\varphi Y],\xi) + \nu\,\big\{(\varphi X)(\eta(Y)) - \eta([\varphi X,Y])\big\}.
\end{align}
Using \eqref{2.1} and equality
$[\varphi Y, \eta(X)\,Q\,\xi] = (\varphi Y)(\eta(X))\,Q\,\xi + \eta(X)[\varphi Y, Q\,\xi]$, we rewrite \eqref{2.6} as
\begin{align}\label{2.8}
% \eta([\varphi^2 X, \,\varphi Y]) + (\varphi X)(\eta(Y))  - \eta([\varphi X,Y])  = 0,
 0= - g([QX,\varphi Y],\xi) +\eta(X)\,\eta([Q\,\xi, \varphi Y]) - (\varphi Y)(\eta(X))\,\nu + (\varphi X)(\eta(Y))\,\nu - \eta([\varphi X,Y])\,\nu.
\end{align}
Equation \eqref{3.11B} gives $\eta([\varphi Y, Q\,\xi])=0$. So, \eqref{2.8} becomes
\begin{align}\label{2.9}
 -\eta([QX, \varphi Y]) - (\varphi Y)(\eta(X))\,\nu + (\varphi X)(\eta(Y)) - \eta([\varphi X,Y])\,\nu
 +\eta(X)\,\eta([\tilde Q\,\xi, \varphi Y]) = 0.
\end{align}
Finally, combining \eqref{2.9} with \eqref{3.6}, we get
\begin{align*}
%\label{3.1KK}
 N^{(2)}(X,Y) =\eta([\tilde QX,\,\varphi Y]) -\eta(X)\,\eta([\tilde Q\,\xi, \varphi Y])
 +g(\tilde Q\,\xi,\xi)
 %(\nu-1)
 \big( (\varphi Y)(\eta(X)) + \eta([\varphi X,Y])\big),
\end{align*}
from which and $X = X^\top +\eta(X)\,\xi$ the required expression of $N^{(2)}$ follows.
\end{proof}

\begin{proposition}\label{prop2.1}
The following equality is valid
{\rm (i)} for a weak contact metric manifold, and
{\rm (ii)}~for a normal weak almost contact metric manifold:
\begin{align}\label{2.Q3}
\iota_{\,\xi}\, d\eta=0,
\end{align}
moreover, the integral curves of $\xi$ are geodesics.	
\end{proposition}

\begin{proof}
(i) Since $N^{(1)}=0$, then $N^{(4)}=0$, see Theorem~\ref{thm6.1}. Thus, \eqref{2.9X} provides the required \eqref{2.Q3}.
(ii)~Equation \eqref{2.3} with $Y=\xi$ yields $d\eta(X,\xi)=g(X,\varphi\,\xi)=0$ for any $X\in\mathfrak{X}_M$; therefore, we get \eqref{2.Q3}.
Using the identity $\pounds=d\circ\iota+\iota\circ d$, from the above we also have
\begin{align}\label{2.Q2}
 \pounds_{\xi}\,\eta = d (\eta(\xi)) + \iota_{\,\xi}\, d\eta = 0.
\end{align}
As in the proof of \cite[Theorem~4.5]{blair2010riemannian}, we obtain $(\pounds_{\xi}\,\eta)(X)= g(X,\nabla_\xi\,\xi)$
for any $X\in\mathfrak{X}_M$. From this and \eqref{2.Q2} we get $\nabla_\xi\,\xi=0$.
%, thus the second claim follows.
\end{proof}

\begin{theorem}\label{thm6.2}
For a weak contact metric structure $(\varphi,Q,\xi,\eta,g)$, the tensors $N^{(2)}$ and $N^{(4)}$ vanish;
moreover, $N^{(3)}$ vanishes if and only if $\,\xi$ is a Killing vector field.
\end{theorem}

\begin{proof} Applying \eqref{2.3} in \eqref{3.6} and using skew-symmetry of $\varphi$ we get $N^{(2)}=0$.
We prove that $N^{(4)}=0$, taking into account \eqref{2.9X} as well as Proposition~\ref{prop2.1}.
%%%%%%%%
Next, we find
\begin{align}\label{3.7}
 (\pounds_{\xi}\,g)(X,Y) &= \xi(g(X,Y)) - g([\xi,X],\varphi Y) - g(X,(\pounds_{\xi}\,\varphi)Y) - g(X,\varphi[\xi,Y]),\\
 (\pounds_{\xi}\,d\eta)(X,Y) &= \xi(d\eta(X,Y)) - d\eta([\xi,X], Y) - d\eta(X,[\xi,Y]). \label{3.8}
\end{align}
Then, invoking the formula \eqref{2.3} in \eqref{3.8} and using \eqref{3.7}, we obtain
\begin{align}\label{3.9}
 (\pounds_{\xi}\,d\eta)(X,Y) = (\pounds_{\xi}\,g)(X,Y) + g(X,(\pounds_{\xi}\,\varphi)Y).
\end{align}
Since $\pounds_V=\iota_{\,V}\circ d+d\circ\iota_{\,V}$, the exterior derivative $d$ commutes with the Lie-derivative, i.e., $d\circ\pounds_V = \pounds_V\circ d$, and using Proposition~\ref{prop2.1}, we get that $d\eta$ is invariant under the action of $\xi$, i.e., $\pounds_{\xi}\,d\eta=0$.
Therefore, \eqref{3.9} implies that $\xi$ is a Killing vector field if and only if $N^{(3)}=0$.
\end{proof}

The following result generalizes \cite[Lemma 6.1]{blair2010riemannian}.

\begin{lemma}\label{lem6.1}
For a weak almost contact metric structure $(\varphi,Q,\xi,\eta,g)$, the covariant derivative of $\varphi$ is given by
\begin{align*}
%\label{3.1}
 2\,g((\nabla_{X}\,\varphi)Y,Z) &= 3\,d\Phi(X,\varphi Y,\varphi Z) - 3\, d\Phi(X,Y,Z) + g(N^{(1)}(Y,Z),\varphi X)\notag\\
 &+ N^{(2)}(Y,Z)\,\eta(X) + 2\,d\eta(\varphi Y,X)\,\eta(Z) - 2\,d\eta(\varphi Z,X)\,\eta(Y) + N^{(5)}(X,Y,Z),
\end{align*}
where we supplement the traditional sequence of tensors $N^{(i)}\ (i=1,2,3,4)$, with a new
skew-symmetric with respect to $Y$ and $Z$ tensor $N^{(5)}(X,Y,Z)$ defined by
\begin{align*}
 N^{(5)}(X,Y,Z) &= (\varphi Z)\,(g(X^\top, \tilde QY)) -(\varphi Y)\,(g(X^\top, \tilde QZ)) \\
 & +\,g([X, \varphi Z]^\top, \tilde QY) - g([X,\varphi Y]^\top, \tilde QZ) \\
 & +\,g([Y,\varphi Z]^\top -[Z, \varphi Y]^\top - \varphi[Y,Z],\ \tilde Q X).
\end{align*}
In particular,
\begin{align}\label{KK}
\nonumber
 N^{(5)}(X,\xi,Z) & = g([\xi,\varphi Z]^\top - \varphi[\xi,Z],\, \tilde Q X),\\
%\label{KK2}
\nonumber
 N^{(5)}(\xi,Y,Z) &= g([\xi, \varphi Z]^\top, \tilde QY) -g([\xi,\varphi Y]^\top, \tilde QZ),\\
 N^{(5)}(\xi,\xi,Z) &= N^{(5)}(\xi,Y,\xi)=0.
\end{align}
%Moreover, if {\rm (\ref{E-Q-lambda}-b)} hold, then $N^{(5)}=0$.
\end{lemma}

\begin{proof}
Using \eqref{3.2} and the skew-symmetry of $\varphi$, one can compute
\begin{align}\label{3.4}
2\,g((\nabla_{X}\,\varphi)Y,Z) =\,& 2\,g(\nabla_{X}(\varphi Y),Z) + 2\,g( \nabla_{X}Y,\varphi Z) \notag\\
=\,& X\,g(\varphi Y,Z) + (\varphi Y)\,g(X,Z) - Z\,g(X,\varphi Y) \notag\\
& +\,g([X,\varphi Y],Z) +g([Z,X],\varphi Y) - g([\varphi Y,Z],X) \notag\\
& +\,X\,g(Y,\varphi Z) + Y\,g(X,\varphi Z) - (\varphi Z)\,g(X,Y) \notag\\
& +\,g([X,Y],\varphi Z) + g([\varphi Z,X],Y) - g([Y,\varphi Z],X).
\end{align}
Using \eqref{2.2}, we obtain
\begin{align}\label{XZ}
\notag
 g(X,Z) &= \Phi(\varphi X, Z) -g(X,\tilde Q Z) +\eta(X)\,\eta(Z) +\eta(X)\,\eta(\tilde Q Z)\\
 &= \Phi(\varphi X, Z) + \eta(X)\,\eta(Z)  - g(X^\top, \tilde QZ).
\end{align}
Thus, and in view of the skew-symmetry of $\varphi$ and applying \eqref{XZ} six times, \eqref{3.4} can be written as
\begin{align*}
%\label{3.5}
& 2\,g((\nabla_{X}\,\varphi)Y,Z) = X\,\Phi(Y, Z) \\
&+ (\varphi Y)\,\big(\Phi(\varphi X, {Z})+\eta(X)\,\eta(Z) \big) - (\varphi Y)\,g(X^\top,\tilde QZ)
%\\ &
- Z\,\Phi(X,Y) \\
&- \Phi([X,\varphi Y],\varphi {Z}) + \eta([X,\varphi Y])\eta(Z) - g([X,\varphi Y]^\top,\tilde QZ)
%\\ &
+\Phi([Z,X],Y) \notag\\
&+ \Phi([\varphi Y,Z],\varphi {X}) - \eta([\varphi Y,Z])\,\eta(X) + g([\varphi Y, Z]^\top, \tilde QX)
%\\ &
+ X\,\Phi(Y,Z)
%\\ &
+ Y\,\Phi(X,Z) \\
& - (\varphi Z)\,\big(\Phi(\varphi X, {Y}) + \eta(X)\,\eta(Y)\big) + (\varphi Z) g(X^\top, \tilde QY)
%\\ &
+ \Phi([X,Y],Z) \\
&+ g(\varphi[\varphi Z,X],\varphi {Y}) + \eta([\varphi Z,X])\,\eta(Y) - g([\varphi Z,X]^\top,\tilde QY)\\
&- g(\varphi[Y,\varphi Z],\varphi {X}) - \eta([Y,\varphi Z])\,\eta(X) + g([Y,\varphi Z]^\top, \tilde QX) .
\end{align*}
Recall the co-boundary formula for exterior derivative $d$ on a $2$-form $\Phi$,
\begin{align}\label{3.3}
 d\Phi(X,Y,Z) &= \frac{1}{3}\,\big\{ X\,\Phi(Y,Z) + Y\,\Phi(Z,X) + Z\,\Phi(X,Y) \notag\\
 &-\Phi([X,Y],Z) - \Phi([Z,X],Y) - \Phi([Y,Z],X)\big\}.
\end{align}
We also have
\begin{eqnarray*}
%\label{Eq-N1}\nonumber
  g(N^{(1)}(Y,Z),\varphi X) = g(\varphi^2 [Y,Z] + [\varphi Y, \varphi Z] - \varphi[\varphi Y,Z] - \varphi[Y,\varphi Z], \varphi X)\\
  = g(\varphi[Y,Z], \tilde Q X) + g([\varphi Y, \varphi Z] - \varphi[\varphi Y,Z] - \varphi[Y,\varphi Z] - [Y,Z], \varphi X).
\end{eqnarray*}
From this and \eqref{3.3} we get the required result.
\end{proof}

According to Theorem~\ref{thm6.2}, on a weak contact metric manifold, we get
\[
\Phi=d\eta,\quad N^{(2)}= N^{(4)}=0.
\]
Thus, invoking \eqref{2.3} and using $d^2=0$ and Proposition~\ref{prop2.1} in Lemma~\ref{lem6.1}, we obtain

\begin{corollary}\label{cor3.1}
For a weak contact metric structure $(\varphi,Q,\xi,\eta,g)$, the covariant derivative of $\varphi$ is given by
\begin{align*}
%\label{3.1A}
 2\,g((\nabla_{X}\,\varphi)Y,Z) = g(N^{(1)}(Y,Z),\varphi X) + 2\,d\eta(\varphi Y,X)\,\eta(Z) - 2\,d\eta(\varphi Z,X)\,\eta(Y) + N^{(5)}(X,Y,Z).
\end{align*}
In particular, we have
\begin{align}\label{3.1AA}
 2\,g((\nabla_{\xi}\,\varphi)Y,Z) &= N^{(5)}(\xi,Y,Z).
\end{align}
\end{corollary}

\begin{example}\rm
Consider a weak almost contact metric structure $(\varphi,Q,\xi,\eta,g)$ with
\[
 Q\,|_{\,{\mathcal D}}=\lambda\,{\rm id}_{\mathcal D}
\]
for a positive constant $\lambda$.
Using $X^\top = X - \eta(X)\,\xi$, we obtain
\begin{equation}\label{E-lambda-N2}
 \tilde QX = (\lambda-1)X^\top +  (\nu-1)\,\eta(X)\,\xi),\quad X\in \mathfrak{X}_M.
\end{equation}
Using \eqref{E-lambda-N2}, we find expressions of the tensor $N^{(5)}$ (see Lem\-ma~\ref{lem6.1})
\begin{align*}
 N^{(5)}(X,Y,Z) & = (\lambda-1)\,\big\{(\varphi Z)\,g(X^\top,Y) - (\varphi Y)\,g(X^\top,Z)
               +g([X,\varphi Z]^\top,Y) \\
&              - g([X,\varphi Y]^\top,Z) + g([Y,\varphi Z]^\top - [Z,\varphi Y]^\top - \varphi[Y,Z],X)\big\},\\
 N^{(5)}(\xi,Y,Z) & =(\lambda-1)\,\big\{g([\xi,\varphi Z]^\top,Y) - g([\xi,\varphi Y]^\top,Z)\big\},\\
 N^{(5)}(X,\xi,Z) & =(\lambda-1)\,\big\{g([\xi,\varphi Z]^\top - \varphi[\xi,Z],X)\big\}.
\end{align*}
\end{example}

\section{The tensor field $h$}
\label{sec:3a}

The tensor field $h=\frac{1}{2}\,\pounds_{\xi}\,\varphi$ plays a important role for contact metric manifolds because of its remarkable properties,
e.g., \cite{blair2010riemannian}. Here we study its generalization.
We define the tensor field $h$ on a weak contact metric manifold similarly as for a contact metric manifold,
\begin{align}\label{4.1}
 h=\frac{1}{2}\, N^{(3)} = \frac{1}{2}\,\pounds_{\xi}\,\varphi.
\end{align}
We compute
\begin{align}\label{4.2}
\nonumber
 (\pounds_{\xi}\,\varphi)X &\overset{\eqref{3.3B}}
% = \pounds_{\xi}(\varphi X) - \varphi (\pounds_{\xi}X)\notag\\ \nonumber
 = \nabla_{\xi}(\varphi X) - \nabla_{\varphi X}\,\xi - \varphi(\nabla_{\xi}X - \nabla_{X}\,\xi)\notag\\
 &= (\nabla_{\xi}\,\varphi)X - \nabla_{\varphi X}\xi + \varphi\nabla_X\, \xi.
\end{align}
Taking $X=\xi$ in \eqref{4.2} and using $\nabla_{\xi}\,\xi=0$ (see Proposition~\ref{prop2.1}) and $(\nabla_{\xi}\,\varphi)\,\xi=\frac12\,N^{(5)}(\xi,\xi,\,\cdot)=0$, see \eqref{3.1AA} and \eqref{KK}, we get
\begin{align}\label{4.2b}
 h\,\xi=0.
\end{align}
From \eqref{4.2}, using
\begin{align*}
%\label{3.1AA}
 2\,g((\nabla_{\xi}\,\varphi)Y,\xi) \overset{\eqref{3.1AA}}= N^{(5)}(\xi,Y,\xi) =0,
\end{align*}
we conclude that the distribution ${\cal D}$ is invariant under $h$.

On a contact manifold, $h$ is a symmetric linear operator that anticommutes with $\varphi$, see \cite[Lemma~6.2]{blair2010riemannian}.
The following lemma generalizes this results.

\begin{lemma}\label{L3.1}
On a weak contact metric manifold, the tensor $h$ satisfies
\begin{eqnarray}
\label{E-31}
 g((h-h^*)X,Y) & = & g([\xi, \varphi Y]^\top, \tilde QX) -g([\xi,\varphi X]^\top, \tilde QY),\\
 \label{E-31A}
 (h\varphi+\varphi h)X &=& \frac12\,\big([\tilde QX,\,\xi] -\tilde Q[X,\,\xi]\big),\\
\label{E-30}
 g(Q\,\nabla_{X}\,\xi, Z) &=& g((\varphi+h\varphi) Z,QX) - \frac 12\,N^{(5)}(X,\xi,\varphi Z),
\end{eqnarray}
where $h^*$ is the conjugate operator to $h$ and the tensor $N^{(5)}(X,\xi,Z)$ is given in \eqref{KK}.
\end{lemma}

\begin{proof}
(i) The scalar product of \eqref{4.2} with $Y$ for $X,Y\in{\cal D}$, using \eqref{4.2b} and Corollary~\ref{cor3.1}, gives
\begin{align}\label{4.3}
 g((\pounds_{\xi}\,\varphi)X,Y) &= \frac{1}{2}\big\{g([\xi, \varphi Y]^\top, \tilde QX) -g([\xi,\varphi X]^\top, \tilde QY)\big\}
 +g(\varphi\nabla_{X}\,\xi - \nabla_{\varphi X}\xi, Y).
\end{align}
Similarly,
\begin{align}\label{4.3b}
 g((\pounds_{\xi}\,\varphi)Y,X) &= \frac{1}{2}\big\{g([\xi, \varphi X]^\top, \tilde QY) -g([\xi,\varphi Y]^\top, \tilde QX)\big\}
 +g(\varphi\nabla_{Y}\,\xi - \nabla_{\varphi Y}\,\xi, X).
\end{align}
The difference of \eqref{4.3} and \eqref{4.3b} gives \eqref{E-31}.

(ii)
Using $\nabla_\xi\,\xi=0$ (see Proposition~\ref{prop2.1}), $\nabla_\xi\,\eta=0$ and $(\nabla_\xi\,\tilde Q)\,\xi=0$, we obtain
\begin{eqnarray*}
 & \varphi(\nabla_\xi\,\varphi)+(\nabla_\xi\,\varphi)\varphi = \nabla_\xi\,(\varphi^2)
 = -\nabla_\xi\,\tilde Q +\nabla_\xi\,(\eta\otimes Q\,\xi)\\
 & = -\nabla_\xi\,\tilde Q +(\nabla_\xi\,\eta)\otimes Q\,\xi +\eta\otimes (\nabla_\xi\,\tilde Q)\,\xi
 = -\nabla_\xi\,\tilde Q .
\end{eqnarray*}
By this and \eqref{4.2}, we get
\begin{eqnarray*}
 2(h\varphi+\varphi h)X & = & \varphi(\pounds_{\xi}\,\varphi)X +(\pounds_{\xi}\,\varphi)\varphi X  \\
 & = & \varphi(\nabla_\xi\,\varphi)X +(\nabla_\xi\,\varphi)\varphi X +\varphi^2\nabla_X\,\xi -\nabla_{\varphi^2 X}\,\xi\\
 & = & -(\nabla_\xi\,\tilde Q)X -\tilde Q\nabla_X\,\xi+\nabla_{\tilde QX}\,\xi,
\end{eqnarray*}
that provides \eqref{E-31A}.
Note that $(h\varphi+\varphi h)\,\xi = 0$.
%%%%%%%%%%%%%%%%%%%%%%%%%%%%%%%%%%%%%%%%%%%%%%%%%

(iii) From Corollary \ref{cor3.1} with $Y=\xi$, we find
\begin{align}\label{4.4}
2\,g((\nabla_{X}\,\varphi)\xi,Z) &=  g(N^{(1)}(\xi,Z),\varphi X) - 2\,d\eta(\varphi Z,X) + N^{(5)}(X,\xi,Z).
\end{align}
From \eqref{2.5} with $Y=\xi$, we get
\begin{align}\label{2.5B}
 [\varphi,\varphi](X,\xi) = \varphi^2 [X,\xi] - \varphi[\varphi X,\xi] .
\end{align}
Using \eqref{2.2}, \eqref{2.5B} and \eqref{3.3B}, we calculate
\begin{align}
\label{4.4A}
 g([\varphi,\varphi](\xi,Z),\varphi X)&= g(\varphi^2\,[\xi,Z] - \varphi[\xi,\varphi Z],\varphi X)\notag\\
 &= - g(\varphi(\pounds_{\xi}\,\varphi Z-\pounds_{\varphi Z}\,\xi),\varphi X)\notag\\
 &= - g(\varphi(\pounds_{\xi}\,\varphi)Z,\varphi X)\notag\\
 &= - g((\pounds_{\xi}\,\varphi)Z,QX) +\eta(X)\,\eta(Q(\pounds_{\xi}\,\varphi)Z) .
 \end{align}
Since $\nabla_{\xi}\,\xi=0$ and $Q\xi=\nu\,\xi$, from \eqref{4.2} we get
\begin{align}\label{4.5}
 g(Q(\pounds_{\xi}\,\varphi)X,\xi) = \nu\,\big\{g((\nabla_{\xi}\,\varphi)X,\xi)-g(\nabla_{\varphi X}\,\xi,\xi) + g(\varphi(\nabla_{X}\,\xi),\xi)\big\}=0.
\end{align}
Since $\varphi\,\xi=0$, we find
\begin{align}\label{4.5A}
 (\nabla_{X}\,\varphi)\,\xi=-\varphi\,\nabla_{X}\,\xi.
\end{align}
Thus, combining \eqref{4.1}, \eqref{4.4}, \eqref{4.4A} and \eqref{4.5}, we deduce
\begin{align}\label{4.6}
 -g(\varphi\,\nabla_{X}\,\xi, Z) = -g(hZ,QX)
 -g(Z,QX) + \eta(Z)\,\eta(QX) + \frac 12\,N^{(5)}(X,\xi,Z) .
\end{align}
Since $Q$ is symmetric and $\varphi$ is skew-symmetric, replacing $Z$ by $\varphi Z$ in \eqref{4.6} and using \eqref{2.1}, $g(\nabla_{X}\,\xi,\xi)=0$ and $\varphi\,\xi=0$, we achieve \eqref{E-30}.
\end{proof}

Set
\begin{align}\label{E-5.0}
 A=h\varphi+\varphi h,\quad
 B= h^* - h.
\end{align}
where operators $A,B$ are represented in Lemma~\ref{L3.1},
%Note that
and $A=B=0$ for a contact structure.
For a weak Sasakian manifold we have $h=0$ (see Theorem~\ref{thm6.1}), hence $A=B=0$ (see \eqref{E-5.0}).

\begin{proposition}
For a weak contact metric structure $(\varphi,Q,\xi,\eta,g)$ on $M^{2n+1}$ we obtain
\begin{align}\label{E-5.1}
 N^{(5)}(\varphi^2 Y,\xi, X) - N^{(5)}(\varphi X,\xi, \varphi Y) = 2\,g( ((h^* - h)\varphi + 2\,\varphi h) X ,\ Y) .
\end{align}
\end{proposition}

\begin{proof}
Using the equality
\begin{align}\label{E-d-eta}
 2\,d\eta(X,Y) = g(\nabla_X\,\xi,\,Y) -g(\nabla_Y\,\xi,\,X),
\end{align}
we find
\begin{align*}
 2\,d\eta(X,Y) - 2\,g(X, \varphi Y) & = g(\nabla_X\,\xi, Y)-g(\nabla_Y\,\xi, X) - 2\,g(X, \varphi Y) \\
                                    & = g(Q\nabla_X\,\xi, \tilde Y)-g(Q\nabla_Y\,\xi, \tilde X) - 2\,g(X, \varphi Y),
\end{align*}
where $X=Q\tilde X$ and $Y=Q\tilde Y$ and $X,Y\in \mathfrak{X}_M$.
From this, using \eqref{E-30} and the equality
\[
 g(\varphi\tilde Y, QX)-g(\varphi\tilde X, QY)=2\,g(X, \varphi Y),
\]
we get
\begin{align}\label{E-5.01}
 0 = 2\,d\eta(X,Y) - 2\,g(X, \varphi Y) & = g((h\varphi)\tilde Y, Q X)-g((h\varphi)\tilde X, Q Y) \notag\\
 & -\frac12\big\{N^{(5)}( X,\xi, \varphi\tilde Y) - N^{(5)}( Y,\xi, \varphi\tilde X)\big\} .
\end{align}
Using \eqref{E-5.0}, we find
\begin{align}\label{E-5.00}
 Qh= -\varphi^2 h = -\varphi(A-h\varphi) = -\varphi A +(A-h\varphi)\varphi = hQ+[A,\varphi].
\end{align}
From \eqref{E-5.01}, using \eqref{E-5.0} and \eqref{E-5.00} (and deleting 'tilde' over $X$ and $Y$), we deduce
\begin{align*}
%\label{E-5.1}
%1 N^{(5)}( X,\xi, \varphi Y) - N^{(5)}( Y,\xi, \varphi X) + 2\,g( (B\varphi+A)\varphi X,\ \varphi Y) = 0.
 N^{(5)}(\varphi^2 Y,\xi, \varphi X) - N^{(5)}(\varphi^2 X,\xi, \varphi Y) = 2\,g( (B\varphi+A +[\varphi, h])\varphi X ,\ Y) .
\end{align*}
From this, using \eqref{E-5.0} and replacing $\varphi X$ by $X$, we get \eqref{E-5.1}.
\end{proof}

\section{Weak Sasakian manifolds}
\label{sec:3}

We are based on some classical results, e.g., \cite[Chapter~6]{blair2010riemannian, sasaki1965almost},
and show the equality $C_3= C^{\,w}_{3/\sim}$.

\begin{lemma}
%\label{thm5.1}
Let $M^{2n+1}(\varphi,Q,\xi,\eta,g)$ be a weak Sasakian manifold. Then
\begin{align}\label{4.10}
 g((\nabla_{X}\,\varphi)Y,Z) &
 = g(QX,Y)\,\eta(Z) - g(QX,Z)\,\eta(Y)
 %g(X,Y)\,\eta(Z) - g(X,Z)\,\eta(Y) \notag\\
 %&+ g(\tilde QX,Y)\,\eta(Z) - g(\tilde QX,Z)\,\eta(Y)
 +\frac{1}{2}\, N^{(5)}(X,Y,Z).
\end{align}
\end{lemma}

\begin{proof} Since $(\varphi,Q,\xi,\eta,g)$ is a weak contact metric structure with $N^{(1)}=0$, by Corollary~\ref{cor3.1}, we~get
\begin{eqnarray}\label{4.11}
 g((\nabla_{X}\,\varphi)Y,Z) = d\eta(\varphi Y,X)\,\eta(Z) - d\eta(\varphi Z,X)\,\eta(Y) + \frac{1}{2}\,N^{(5)}(X,Y,Z).
\end{eqnarray}
Using \eqref{2.3} and \eqref{2.2} in \eqref{4.11}, we get \eqref{4.10}.
%%%%%%%%%
\end{proof}

%%%%%%%%%%% NEW
Since on a weak contact metric manifold, the tensor $N^{(3)}$ vanishes if and only if $\xi$ is Killing (Theorem~\ref{thm6.2}),
then a weak almost contact metric structure is weak $K$-contact if and only if $h=0$.
Thus we have the following generalization of \cite[Corollary~6.3]{blair2010riemannian}.

\begin{proposition}\label{P-4.1}
 A weak Sasakian manifold is weak $K$-contact.
\end{proposition}

\begin{proof}
In view of \eqref{4.5A}, Eq. \eqref{4.10} with $Y=\xi$ becomes
\begin{align}\label{4.6A}
g(\nabla_{X}\,\xi, \varphi Z) = \eta(X)\,\eta(QZ)
-g(X,QZ) + \frac 12\,N^{(5)}(X,\xi,Z).
\end{align}
Combining \eqref{4.6} and \eqref{4.6A}, we achieve $g(hZ,QX)=0$, which implies $h=0$.
%Hence the proof.
\end{proof}
%%%%%%%%%%%%%+

The main result in this section is the following rigidity of a Sasakian structure.

\begin{theorem}\label{T-4.1}
A weak almost contact metric structure $(\varphi,Q,\xi,\eta,g)$ on $M^{2n+1}$ is weak Sasakian if and only if
it is homothetically equivalent to a Sasakian structure $(\varphi',\xi,\eta,g')$ on $M^{2n+1}$.
\end{theorem}

\begin{proof}
Let $M^{2n+1}(\varphi,Q,\xi,\eta,g)$ be a weak Sasakian manifold. Replacing $Y$ by $\xi$ in \eqref{4.10}, we get
\begin{align}\label{4.12}
\nonumber
 g((\nabla_{X}\,\varphi)\,\xi,Z) &= \eta(QX)\,\eta(Z) - g(QX,Z) + \frac{1}{2}\,N^{(5)}(X,\xi,Z) \\
 & = - g(Q X^\top, Z) +\frac{1}{2}\,N^{(5)}(X,\xi,Z).
\end{align}
Further, $\varphi\,\xi=0$ gives us $(\nabla_{X}\,\varphi)\,\xi=-\varphi\,\nabla_{X}\,\xi$.
Thus, \eqref{4.12} can be written as
\begin{align*}
%\label{4.13}
 g(\nabla_{X}\,\xi,\varphi Z) = - g(Q X^\top,Z) +\frac{1}{2}\, N^{(5)}(X,\xi,Z).
\end{align*}
By the above and \eqref{2.1}, we find
\begin{align}\label{4.14}
 g(\nabla_X\,\xi +\varphi X^\top, \,\varphi\,Z) = \frac12\,N^{(5)}(X,\xi,Z).
\end{align}
Take $\tilde X,\tilde Y\in{\cal D}$ such that $\varphi\tilde X = X^\top$ and $\varphi\,\tilde Y=Y^\top$.
Using \eqref{E-d-eta} and keeping in mind \eqref{4.14} and \eqref{2.3}, we get
\begin{align}\label{4.15}
 d\eta(X,Y) - g(X, \varphi Y) = \frac14\,\big\{N^{(5)}(\varphi\tilde X,\xi, \tilde Y) - N^{(5)}(\varphi\tilde Y,\xi, \tilde X)\big\}.
\end{align}
By \eqref{4.15}, we conclude that \eqref{2.3} holds if
the bilinear form $(X,Y) \to N^{(5)}(\varphi X,\xi, Y)$ is symmetric:
\begin{align}\label{4.16}
 & N^{(5)}(\varphi X,\xi, Y) = N^{(5)}(\varphi Y,\xi, X) .
 %\quad X,Y\in \mathfrak{X}_M,
\end{align}
%%%%%%%%%%%%%%%%%
Since $\varphi$ is skew-symmetric, applying \eqref{4.10} with $Z=\xi$ in \eqref{4.NN}, we obtain
\begin{align}\label{4.17}
& g( [\varphi,\varphi](X,Y),\xi) = - g((\nabla_{\varphi Y}\,\varphi)X, \xi)  +g((\nabla_{\varphi X}\,\varphi)Y, \xi) \notag\\
& = -g(Q\,\varphi Y,X) + g(Q\,\varphi Y, \xi)\,\eta(X) - \frac{1}{2}\, N^{(5)}(\varphi Y, X, \xi)\notag\\
&  +g(Q\,\varphi X,Y) -g(Q\,\varphi X, \xi)\,\eta(Y)+\frac{1}{2}\,N^{(5)}(\varphi X, Y, \xi).
\end{align}
Recall that $[Q,\,\varphi]=0$. Thus, \eqref{4.17} yields
\begin{align*}
%\label{4.18}
  g( [\varphi,\varphi](X,Y), \xi) = - 2\,g(QX,\varphi Y)
   +\frac{1}{2}\big\{ N^{(5)}(\varphi Y, \xi, X) -N^{(5)}(\varphi X, \xi, Y) \big\}.
\end{align*}
From this, using \eqref{E-tildeQ} and definition of $N^{(1)}$, we get
\begin{align}\label{4.18}
& g(N^{(1)}(X,Y), \xi)  = -2\,g(X,\varphi Y)\,g(\tilde Q\,\xi,\,\xi) -2\,g(\tilde Q X, \varphi Y) \notag\\
& + \frac{1}{2}\big\{ N^{(5)}(\varphi Y, \xi, X) - N^{(5)}(\varphi X, \xi, Y) \big\}.
\end{align}
By \eqref{4.18}, if $N^{(1)}(X,Y)=0$ holds then the following condition is valid for all $X,Y\in \mathfrak{X}_M$:
\begin{align}\label{4.19}
  N^{(5)}(\varphi Y, \xi, X) -N^{(5)}(\varphi X, \xi, Y) = 4\,g(X,\varphi Y)\,g(\tilde Q\,\xi,\,\xi) +4\,g(\tilde Q X, \varphi Y) .
\end{align}
In view of \eqref{4.16} and $g(\tilde Q\,\xi,\,\xi)=\nu-1$, equation \eqref{4.19} reduces to
\begin{align*}
%\label{4.20B}
 & 0 = g(X,\varphi Y)\,g(\tilde Q\,\xi,\,\xi) +g(\tilde Q X, \varphi Y)
     = g(\nu X - QX, \varphi Y),
\end{align*}
thus, $Q\,|_{\cal D}=\nu\,{\rm id}|_{\cal D}$.
By Proposition~\ref{P-22}(iii), our structure is homothetically equivalent to a Sasakian structure.
\end{proof}

\begin{remark}\rm
 In~view of \eqref{2.8X}, \eqref{E2-tildeQ} and \eqref{KK}, we get
\[
 N^{(5)}(\varphi X,\xi, Y) = -g(\tilde Q\,\varphi N^{3}(Y), \,X) = -g((\varphi^3+\varphi)\,h(Y), \,X).
\]
Thus, \eqref{4.16} holds if and only if the linear operator $(\varphi^3+\varphi)\,h:TM\to TM$ is self-adjoint.

One can get Proposition~\ref{P-4.1} as a consequence of Theorem~\ref{T-4.1} and Proposition~\ref{P-22}.
\end{remark}

%almost coK\"ahler
\section{Weak cosymplectic manifolds}
\label{sec:4}

An important class of almost contact manifolds is given by cosymplectic manifolds, e.g., \cite{Cappelletti-Montano2013,goldberg1969integrability}.
Here, we study wider classes of weak (almost) cosymplectic manifolds and characterize the class $C^{\,w}_{4}$ in $C_{4}$ by the condition $\nabla\varphi=0$.

\begin{definition}\rm
A \textit{weak almost cosymplectic} (or a \textit{weak almost coK\"ahler}) manifold is a weak almost contact metric manifold $M^{2n+1}(\varphi,Q,\xi,\eta,g)$, whose fundamental $2$-form $\Phi$ and the 1-form $\eta$ are closed. If a weak almost cosymplectic structure is normal,
we say that $M$ is a \textit{weak cosymplectic} (or a~\textit{weak coK\"ahler}) \textit{manifold}.
\end{definition}

\begin{theorem}\label{thm6.2C}
For a weak almost cosymplectic structure $(\varphi,Q,\xi,\eta,g)$, the tensors $N^{(2)}$ and $N^{(4)}$ vanish.
Moreover, $N^{(1)}=[\varphi,\varphi]$ and $N^{(3)}$ vanishes if and only if $\,\xi$ is a Killing vector field.
\end{theorem}

\begin{proof}
By \eqref{3.6} and \eqref{2.9X} and since $d\eta=0$, the tensors $N^{(2)}$ and $N^{(4)}$ vanish on a weak almost cosymplectic structure.
Moreover, from \eqref{2.6X} and \eqref{3.9}, respectively, the tensor $N^{(1)}$ coincides with the Nijenhuis tensor of $\varphi$,
and $N^{(3)}=\pounds_{\xi}\,\varphi$ vanishes if and only if $\xi$ is a Killing vector.
\end{proof}

By application of the above theorem, we make the following corollary to point out some important attributes of a weak almost cosymplectic structure.

\begin{corollary}
In any weak almost cosymplectic manifold the integral curves of $\,\xi$ are geodesics.
\end{corollary}

\begin{proof}
By Theorem~\ref{thm6.2C}, $N^{(4)}=0$; thus, from \eqref{2.9X} and $g(\nabla_{X}\xi,\,\xi)=0$ we get for any $X\in\mathfrak{X}_M$,
\begin{equation*}
 0=\xi(\eta(X)) - \eta(\pounds_{\xi}X) = \xi(\eta(X)) - g(\nabla_{\xi}X,\xi) - g(\nabla_{X}\xi,\xi) = g(X,\nabla_{\xi}\,\xi).
\end{equation*}
\end{proof}

\begin{proposition}
%Lemma~\ref{lem6.1},
Let $(\varphi,Q,\xi,\eta,g)$ be a weak cosymplectic structure. Then
\begin{align}\label{6.1}
 2\,g((\nabla_{X}\,\varphi)Y,Z) &= N^{(5)}(X,Y,Z),\\
\label{6.1b}
 6\,d\Phi & = N^{(5)}(X,Y,Z) + N^{(5)}(Y,Z,X) + N^{(5)}(Z, X, Y) ,\\
\label{6.1c}
 2\,g([\varphi,\varphi](X,Y),Z) & = N^{(5)}(\varphi X,Y,Z) + N^{(5)}(\varphi Y,Z,X) + N^{(5)}(\varphi Z, X, Y).
\end{align}
\end{proposition}

\begin{proof}
For a weak almost cosymplectic structure $(\varphi,Q,\xi,\eta,g)$, we get
\begin{equation}\label{6.1a}
 2\,g((\nabla_{X}\,\varphi)Y,Z)= g([\varphi,\varphi](Y,Z),\varphi X) + N^{(5)}(X,Y,Z).
\end{equation}
From \eqref{6.1a}, using condition $[\varphi,\varphi]=0$ we get \eqref{6.1}.
Using \eqref{3.3} and \eqref{6.1}, we can write
\[
  3\,d\Phi  = g((\nabla_X\,\varphi)Z,Y) +g((\nabla_Y\,\varphi)X, Z) +g((\nabla_Z\,\varphi)Y, X),
\]
hence, \eqref{6.1b}.
Using \eqref{4.NN}, \eqref{6.1} and the skew-symmetry of $\varphi$, we obtain
\[
 2\,g([\varphi,\varphi](X,Y),Z) = N^{(5)}(X, Y, \varphi Z) {+} N^{(5)}(\varphi X, Y, Z) {-} N^{(5)}(Y, X, \varphi Z) {-} N^{(5)}(\varphi Y,X,Z) .
\]
This and \eqref{6.1b} with $X$ replaced by $\varphi X$ provide \eqref{6.1c}.
\end{proof}

\begin{remark}\rm
For a weak cosymplectic structure, using \eqref{6.1}, we obtain (compare with \eqref{E-30})
\begin{align*}
%\label{6.2}
 2\,g(Q(\nabla_{X}\,\xi),\,Z) = -N^{(5)}(X,\xi,\varphi Z).
\end{align*}
\end{remark}

Recall that an almost contact metric structure $(\varphi,\xi,\eta,g)$ is cosymplectic if and only if $\varphi$ is parallel, e.g., \cite[Theorem~6.8]{blair2010riemannian}.
The following our theorem completes this result.

\begin{theorem}\label{thm6.2D}
Any weak almost contact structure $(\varphi,Q,\xi,\eta,g)$ with the property $\nabla\varphi=0$
is a~weak cosymplectic structure with vanishing tensor $N^{(5)}$.
\end{theorem}

\begin{proof}
Using condition $\nabla\varphi=0$, from \eqref{4.NN} we obtain $[\varphi,\varphi]=0$.
Hence, from \eqref{2.6X} we get $N^{(1)}(X,Y)=2\,d\eta(X,Y) Q\xi$, and from \eqref{4.NNxi} we obtain
\begin{align}\label{E-cond1}
 \nabla_{\varphi X}\,\xi - \varphi\,\nabla_{X}\,\xi = 0,\quad X\in \mathfrak{X}_M.
\end{align}
From \eqref{3.3}, we calculate
\[
 3\,d\Phi(X,Y,Z) = g((\nabla_{X}\,\varphi)Z, Y) + g((\nabla_{Y}\,\varphi)X,Z) + g((\nabla_{Z}\,\varphi)Y,X);
\]
hence, using condition $\nabla\varphi=0$ again, we get $d\Phi=0$. Next,
\begin{align*}
 N^{(2)}(Y,\xi) = -\eta([\varphi Y,\xi]) = g(\xi, \varphi\nabla_\xi Y) =0.
\end{align*}
Thus, setting $Z=\xi$ in Lemma~\ref{thm6.1} and using the condition $\nabla\varphi=0$ and the properties
$d\Phi=0$, $N^{(2)}(Y,\xi)=0$ and $N^{(1)}(X,Y)=2\,d\eta(X,Y) Q\xi$, we find $0 = d\eta(\varphi Y, X) - N^{(5)}(X,\xi, Y)$.
 By \eqref{KK} and \eqref{E-cond1}, we get
\[
 N^{(5)}(X,\xi, Y) = g([\xi,\varphi Y]^\top -\varphi[\xi,Y],\, \tilde Q X)
 = g(\nabla_{\varphi Y}\,\xi - \varphi\,\nabla_{Y}\,\xi,\, \tilde Q X) = 0;
\]
hence, $d\eta=0$. By the above, $N^{(1)}=0$.
Thus, $(\varphi,Q,\xi,\eta,g)$ is a weak cosymplectic structure.
Finally, from \eqref{6.1} and condition $\nabla\varphi=0$ we get $N^{(5)}=0$.
\end{proof}

\begin{example}\rm
% Direct product.
Let $M$ be a $2n$-dimensional smooth manifold and $\tilde\varphi:TM\to TM$ an endomorphism of rank $2n$ such that
$\nabla\tilde\varphi=0$.
To construct a weak cosymplectic structure on $M\times \mathbb{R}$ or $M\times S^1$, take any point $(x, t)$ of either
space and set $\xi = (0, d/dt)$, $\eta =(0, dt)$~and
\[
 \varphi(X, Y) = (\tilde\varphi X, 0),\quad
 Q(X, Y) = (-\tilde\varphi^2 X, \nu\,\xi).
\]
where $X\in M_x$, $Y\in \{\mathbb{R}_t$ or $S^1_t\}$ and $\nu\in\mathbb{R}_+$.
Then \eqref{2.1} holds and Theorem~\ref{thm6.2D} can be applied.
\end{example}

\section{Weak contact vector fields}
\label{sec:5}

Contact vector fields (and contact infinitesimal transformations) is a fruitful tool in contact manifold geometry.
Here, we are based on some classical results, e.g., \cite[Section~5.2]{blair2010riemannian}.

\begin{definition}\rm
%A diffeomorphism on a weak contact metric manifold
%$M^{2n+1}(\varphi,Q,\xi,\eta)$
%is said to be a \textit{weak contact transformation} if there exists a non-vanishing function $\tau:M\to\mathbb{R}$ such that $f^{*}\eta=\tau\,\eta$.
%In particular, if $\tau=1$ then $f$ is called a {\it strict weak contact transformation}.
 A vector field $X$ on a weak contact metric manifold $M^{2n+1}(\varphi,Q,\xi,\eta)$ is called
 a {\it weak contact vector field}
(or, a {\it weak contact infinitesimal transformation}), if
the flow of $X$ preserves $\xi$, i.e., there exists a smooth function $\sigma : M \rightarrow \mathbb{R}$ such that
\begin{align}\label{7.1}
\pounds_{X}\,\eta=\sigma\,\eta,
\end{align}
and if $\sigma=0$, then the vector field $X$ is said to be \textit{strict weak contact vector field}.
\end{definition}

The following our result generalizes \cite[Theorem~5.7]{blair2010riemannian}.

\begin{theorem}
%\label{thm7.3}
A vector field $X$ on a weak contact metric manifold $M^{2n+1}(\varphi,Q,\xi,\eta,g)$ is a weak contact infinitesimal transformation
if and only if there exists a function $f$ on $M$ such that
\begin{align}\label{7.1b}
 QX = -\frac{1}{2}\,\varphi\,\nabla f +\nu\,f\,\xi.
\end{align}
Moreover, a weak contact vector field $X$ is strict if and only if $\,\xi(f)$.
\end{theorem}

\begin{proof}
One can explicitly write $(\pounds_{X}\,\eta)(Y)=X(\eta(Y)) - \eta([X,Y])$. Using \eqref{3.3A}, we rewrite \eqref{7.1} as
\begin{align}\label{7.2}
 2\,d\eta(X,Y)+Y(f)=\sigma\,\eta(Y),
\end{align}
where $f=\eta(X)$. Using \eqref{2.3} in \eqref{7.2}, we get
%\[
$-2\,g(\varphi X,Y)+Y(f)=\sigma\,\eta(Y)$,
%\]
which provides
\begin{align}\label{7.2b}
 -2\,\varphi X+\nabla f=\sigma\,\xi .
\end{align}
Applying $\varphi$ and using \eqref{2.1}, we obtain \eqref{7.1b}.
Multiplying \eqref{7.2b} by $\xi$ gives $\sigma=\xi(f)$.

Conversely, assume the condition \eqref{7.1b}.
The scalar product of \eqref{7.1b} with $\xi$ gives
\[
 \nu\,\eta(X)=g(Q\,\xi, X) = \nu\,f;
\]
therefore, $f=\eta(X)$.
Multiplying \eqref{7.1b} by $\varphi$ and using \eqref{2.1}, gives
\[
 Q\,\varphi X =-\frac12\,\varphi^2(\nabla f) = \frac12\,Q\,\nabla f - \frac12\,\eta(\nabla f)\,Q\,\xi,
\]
then, using non-degeneracy of $Q$, we get
\begin{align*}
 \nabla f - 2\,\varphi X = \eta(\nabla f)\,\xi = \xi(f)\,\xi.
\end{align*}
Thus, using \eqref{2.2}, \eqref{2.3} and \eqref{3.3A}, one can calculate
\begin{align*}
(\pounds_{X}\,\eta)(Y)&=2\,d\eta(X,Y)+Y(\eta(X)) = 2\,d\eta(QX-\tilde QX,Y) + Y(f)\\
&= 2\,g(-\frac{1}{2}\,\varphi \nabla f +f\,\nu\,\xi,\varphi Y) + Y(f) -2\,g(\tilde QX,\varphi Y)\\
&= -g(\varphi \nabla f,\varphi Y) + Y(f) -2\,g(\tilde QX,\varphi Y)\\
&= -(QY)(f)+\nu\,\xi(f)\,\eta(Y) + Y(f) -2\,g(\tilde QX,\varphi Y)\\
%&= \nu\,\xi(f)\,\eta(Y) -(\tilde QY)(f) -2\,g(\tilde QX,\varphi Y)\\
&= \nu\,\xi(f)\,\eta(Y) -g(\tilde Q(\nabla f - 2\,\varphi X), Y)\\
&= \nu\,\xi(f)\,\eta(Y) -(\nu-1)\,\xi(f)\,\eta(Y) = \xi(f)\,\eta(Y).
\end{align*}
By this, \eqref{7.1} with $\sigma = \xi(f)$ is valid.
\end{proof}


\begin{thebibliography}{00}

\bibitem{blair2010riemannian}
D. Blair, {\it Riemannian geometry of contact and symplectic manifolds}, Springer, 2010

\bibitem{bg}
C.\,P. Boyer, and K. Galicki, {\it Sasakian Geometry}, Oxford University Press, 2008

\bibitem{Cappelletti-Montano2013}
B. Cappelletti-Montano, A. De Nicola, and I. Yudin, A survey on cosymplectic geometry.
Rev. Math. Phys. 25, No. 10, Article ID 1343002, 55 p. (2013).

\bibitem{RWo-2}
V. Rovenski, and R. Wolak,
{New metric structures on $\mathfrak{g}$-foliations},
Indagationes Mathema\-ticae, 2021, https://doi.org/10.1016/j.indag.2021.11.001
%{\it The partial Ricci flow on $\mathfrak{g}$-foliations}, arXiv:1905.07704, 2019, 16\,pp.

\bibitem{Rov-Wa-2021}
V. Rovenski, and P. G. Walczak,
\textit{Extrinsic geometry of foliations}. Progress in Mathematics, vol. 339, Birkh\"{a}user, 2021

\bibitem{goldberg1969integrability}
S.\,I. Goldberg and K. Yano, {Integrability of almost cosymplectic structures}, Pacific Journal of Mathematics, {31}\,(2), 373--382 (1969)

\bibitem{sasaki1965almost}
S. Sasaki, {\it Almost contact manifolds, Lecture notes}, Tohoku University, 1965

\bibitem{kw1985}
S. Kaneyuki and F.\,L. Willams, Almost paracontact and parahodge structures on manifolds. Nagoya Math. J. 99, 173--187 (1985)

\end{thebibliography}
\end{document}